\documentclass[11pt,twoside,a4paper]{article}

\usepackage{graphicx,url}
\usepackage{setspace}
\usepackage{enumerate,lineno,setspace,float}
\usepackage[small,compact]{titlesec}
\usepackage{amsmath,amsfonts,amssymb,amsthm}
\usepackage{geometry}
\usepackage{fancyhdr}
\usepackage{amssymb,amsmath}
\usepackage{latexsym}
\usepackage{graphicx,float}
\usepackage{url}
\usepackage[capitalize]{cleveref} 
\usepackage{tikz}
\usepackage{caption}
\usepackage[numbers,sort&compress]{natbib}

\newtheorem{Theorem}{Theorem}[section]

\newtheorem{Proposition}[Theorem]{Proposition}
 \newtheorem{Corollary}[Theorem]{Corollary}
\newtheorem{Lemma}[Theorem]{Lemma}

\newtheorem{Observation}[Theorem]{Observation}

\newtheorem{Question}{Question}

\usepackage{color}

\definecolor{VeryLightBlue}{rgb}{0.9,0.9,1}
\definecolor{LightBlue}{rgb}{0.8,0.8,1}
\definecolor{MidBlue}{rgb}{0.5,0.5,1}
\definecolor{DarkBlue}{rgb}{0,0,0.6}
\definecolor{Blue}{rgb}{0,0,1}
\definecolor{Gold}{rgb}{1,0.843,0}
\definecolor{LightGreen}{rgb}{0.88,1,0.88}
\definecolor{MidGreen}{rgb}{0.6,1,0.6}
\definecolor{DarkGreen}{rgb}{0,0.6,0}
\definecolor{VeryLightYellow}{rgb}{1,1,0.9}
\definecolor{LightYellow}{rgb}{1,1,0.6}
\definecolor{MidYellow}{rgb}{1,1,0.5}
\definecolor{DarkYellow}{rgb}{1,1,0.2}
\definecolor{DarkPurple}{rgb}{.6,0,1}
\definecolor{Red}{rgb}{1,0,0}
\definecolor{VeryLightRed}{rgb}{1,0.9,0.9}
\definecolor{LightRed}{rgb}{1,0.8,0.8}
\definecolor{MidRed}{rgb}{1,0.55,0.55}

\def\leq{\leqslant}
\def\geq{\geqslant}

\long\def\delete#1{}

\input amssym.def
\input amssym.tex

\newcommand{\be}{\begin{equation}}
\newcommand{\ee}{\end{equation}}
\newcommand{\bea}{\begin{eqnarray}}
\newcommand{\eea}{\end{eqnarray}}
\newcommand{\bean}{\begin{eqnarray*}}
\newcommand{\eean}{\end{eqnarray*}}

\def\diam{d}

\def\ve{\varepsilon}

\def\mt{\mathcal}
\def\L{\mathcal{L}}
\def\({\left(}
\def\){\right)}
\def\[{\left[}
\def\]{\right]}

\begin{document}

\title{Optimal Radio Labellings of Block Graphs and Line Graphs of Trees}

\author{\textbf{Devsi Bantva} \\ Department of Mathematics \\ Lukhdhirji Engineering College, Morvi - 363 642 \\ Gujarat, India \\ E-mail : \textit{devsi.bantva@gmail.com} \\ \\
\textbf{Daphne Der-Fen Liu} 
\thanks{Research is partially supported by the National Science Foundation grant DMS 1600778 and the NASA MIRO NX15AQ06A grant.} 
\\
California State University Los Angeles, USA \\ E-mail : \textit{dliu@calstatela.edu}
}

\pagestyle{myheadings}
\markboth{\centerline{Devsi Bantva and Daphne Der-Fen Liu}}{\centerline{Optimal radio labellings of block graphs and line graphs of trees}}
\date{}
\openup 0.8\jot
\maketitle

\begin{abstract}
A radio labeling of a graph $G$ is a mapping $f$ : $V(G)$ $\rightarrow$ \{0, 1, 2,...\} such that $|f(u)-f(v)| \geq \diam(G) + 1 - d(u,v)$ holds for every pair of vertices $u$ and $v$, where $\diam(G)$ is the diameter of $G$ and $d(u,v)$ is the distance between $u$ and $v$ in $G$. The radio number of $G$, denoted by $rn(G)$, is the smallest $t$ such that $G$ admits a radio labeling with  $t=\max\{|f(v)-f(u)|: v, u \in V(G)\}$. A block graph is a graph such that each block (induced maximal 2-connected subgraph) is a complete graph.  In this paper, a lower bound for the radio number of block graphs is established. The block graph which achieves   this bound is called a {\it lower bound block graph}. We prove  three necessary and sufficient conditions for lower bound block graphs.  Moreover, we give three sufficient conditions for a graph to be a lower bound block graph. Applying the established bound and conditions, we show that several families of block graphs are lower bound block graphs, including the level-wise regular block graphs and the extended star of blocks. The line graph of a graph $G(V,E)$ has $E(G)$ as the vertex set, where two vertices are adjacent if they are incident edges in $G$. 
We extend our results to trees as trees and its line graphs are block graphs. We prove that if a tree is a lower bound block graph then, under certain conditions, its line graph is also a lower bound block graph, and vice versa. Consequently, we show that the line graphs of many known lower bound trees, excluding paths, are lower bound block graphs.  

\medskip
\noindent
\emph{Keywords}: Radio labeling, radio number, tree, block graph, level-wise regular  block graph, extended star of blocks, line graph. \\

\noindent
\emph{AMS Subject Classification (2010)}: 05C78, 05C15, 05C12.
\end{abstract}

\section{Introduction} 

The radio labeling is motivated by the channel assignment introduced by Hale \cite{Hale}. The task is to assign channels to some given transmitters or stations such that the interference is avoided and the spectrum of the channels used is minimized. The interference between two transmitters is closely related to the proximity of the locations, the closer the location the stronger the interference might occur. In order to avoid stronger interference, the larger their separation in assigned frequencies must be.

In a graph model of the channel assignment, the transmitters are represented by  vertices, and two transmitters are connected by an edge if they are close to each other.  The distance between two vertices $u, v$ in $G$,  denoted by $d(u,v)$, is the length of a shortest $(u,v)$-path. The diameter of $G$, denoted by $d(G)$, is the maximum distance between two vertices in $G$.  A {\it radio $k$-labeling} is a function $f: V(G) \to \{0,1,2,\ldots\}$ such that the following holds for all $u, v \in V(G)$: 
\begin{equation}\label{def:rn}
|f(u) - f(v)| \geq k+1 - d(u,v). 
\end{equation}
The {\it span} of a radio $k$-labeling $f$, denoted by $span(f)$, is the $\max\{ | f(u)-f(v)|: u, v \in V(G)\}$. The {\it radio-$k$-number}, denoted by $rn_k(G)$, is the minimum span of a radio $k$-labeling admitted by $G$.  

When $k=2$, a radio $2$-labeling is known as {\it distance two labeling} or {\it $L(2,1)$-labeling}. In this special case, $rn_2(G)$ is termed as the $\lambda$-number of graph $G$ denoted by $\lambda(G)$. The $L(2,1)$-labeling has been  studied extensively over the past years (cf. \cite{Griggs, Kral}). For more details and bibliographic references, the readers are referred to surveys by Calamoneri \cite{Calamoneri} and by Yeh \cite{Yeh1}.

In this paper, we focus on another special case when  $k=d(G)$. The radio $d(G)$-labeling is known as the {\it radio labeling} of $G$. In this case, $rn_{d(G)} (G)$ is denoted by $rn(G)$. Introduced by Chartrand et al. \cite{Chartrand1, Chartrand2},   radio labeling has been studied extensively in the past two decades (cf. \cite{Bantva1,Bantva2,Chartrand1,Chartrand2,Daphne1,Liu2, MCC, graceful,JP, LP1,  Vaidya1,Vaidya2,Vaidya3,Zhou1, Zhou3, Das}). The radio number of cycles and paths were determined by Liu and Zhu \cite{Liu2}. Khennoufa and Togni \cite{Khennoufa} studied the radio number for hypercubes by using generalized binary Gray codes.  
Martinez et al. \cite{JP} investigated the radio number of generalized prism graphs. Niedzialomski \cite{graceful} studied radio graceful graphs (where $G$ admits a surjective radio labeling) and showed that the Cartesian  product of $t$ copies of a complete graph is radio graceful for certain $t$, providing infinitely many examples of radio graceful graphs of arbitrary diameter.  For two positive integers $m, n  \geqslant 3$, the {\it toroidal grid} $G_{m,n}$ is the Cartesian product of cycles $C_m$ and $C_n$, $C_{m}\Box C_{n}$. Morris et al. \cite{MCC} determined $rn(G_{n,n})$, and Saha and Panigrahi \cite{LP2} determined $rn(G_{m,n})$ when $mn\equiv0\pmod2$. 

The radio number for trees has been investigated by many authors. Liu \cite{Daphne1} proved a general lower bound of the radio number for trees and characterized the spiders (trees with at most one vertex of degree greater than 2) whose  radio numbers achieve this bound. We call trees whose  radio numbers achieve this bound as {\it lower bound trees}.  Many  families of trees have been shown to be lower bound  trees, including  complete level-wise regular trees without degree 2 vertices (Hal\'asz and Tuza \cite{Tuza}), complete $m$-ary trees with $m \geq 3$ (Li et al. \cite{Li}), and banana trees and firecrackers trees (Bantva et al. \cite{Bantva1, Bantva2}). Most recently, this lower bound has been improved by Liu, Saha, and Das  \cite{ImprovedBound} for non-lower-bound trees. 

The aim of this article is to extend the study on trees to block graphs. For a graph $G$, a {\it 2-connected induced subgraph} of $G$, denoted by $G[V']$, is a subgraph induced by the vertex set $V' \subseteq V(G)$, where $G[V']$ does not contain any cut-vertex. We call $G[V']$  {\it maximal} if  $G[W]$ is not a 2-connected induced subgraph of $G$ for any superset $W \supset V'$. A {\it block graph} is a graph whose maximal 2-connected induced subgraphs are complete graphs (cliques). Obviously, a tree is a block graph.  Generalizing the approach used in \cite{Daphne1} we prove a general lower bound for the radio number of block graphs (\cref{thm:lb}). A block graph  whose radio number is equal to this bound is called a {\it lower bound block graph}.  We prove three necessary and sufficient conditions for lower bound block graphs  (Theorems \ref{thm:lb}, \ref{thm:main}, \ref{thm:main2}).  Moreover, 
three sufficient conditions have been derived for a block graph to be a lower bound block graph (\cref{thm:suf}). 

Applying these results, in Section 4, we present two  families of lower bound block graphs, namely, the level-wise regular block graphs and the extended star of blocks. In Section 5, we discuss relations between the radio numbers of a tree and its line graph (both are block graphs). In particular, we prove sufficient conditions for lower bound trees so that their line graphs are lower bound block graphs, and vice versa. 
These results imply that the line graphs of the above mentioned known lower bound trees (complete level-wise regular trees without degree 2 vertices, complete $m$-ary trees, $m \geq 3$, banana trees, and firecrackers trees,  etc.) are all lower bound block graphs.


 \section{Preliminaries}


This section is aimed to provide preliminaries needed for the advancement of discussion. The distance between two vertices $u$ and $v$ in a graph $G$, written as $d_{G}(u,v)$, is the length of a shortest path between $u$ and $v$. We denote $d_{G}(u,v)$ by $d(u,v)$ when $G$ is clear in the context.  A $(u,v)$-path with length $d(u,v)$ is called a {\it $(u,v)$-geodesic}. The diameter of a graph $G$, denoted by $\diam(G)$, is the maximum $d(u,v)$ among all pairs of vertices.   

Let $G$ be a simple graph. A vertex $v \in V(G)$ is called a {\it cut-vertex} if its deletion increases the number of components of $G$. A graph is {\it 2-connected} if it does not have any cut-vertices. A {\it block} of $G$ is a maximal 2-connected induced subgraph of $G$. A {\it block graph} is a connected graph so that all its blocks are complete graphs (cliques). Note that a tree is a block graph where each block is a single edge. For a block graph $G$, there exists a unique $(u, v)$-geodesic between any two  vertices $u$ and $v$.  Throughout the article, all graphs considered are block graphs. 

Let $G$ be a block graph. For a vertex $v \in V(G)$, the {\it weight of $G$ at $v$} is defined by $wt(v)$ = $\sum_{u \in V(G)}d(u,v)$. The {\it weight of $G$} is 
$$
wt(G) = \min\{wt(v) : v \in V(G)\}.
$$
A vertex $v \in V(G)$ is called a {\it weight center} of $G$ if $wt(v)$ = $wt(G)$. The set of weight centers of $G$ is denoted by $W(G)$. 

For two vertices $u$ and $v$ in a graph $G$,  denote $A(v,u)$ the set of vertices  that are closer to $v$ than to $u$.  That is,
$$
A(v,u) = \{z \in V(G): d(v, z) < d(u, z)\}. 
$$
Denote $n(v,u) = |A(v, u)|$.  

\begin{Proposition}
\label{temp}
Let $G$ be a block graph. The following hold:  
\begin{enumerate}[\rm (1)]
    \item If $a, w, u \in V(G)$ where $aw, wu \in E(G)$ but $au \not\in E(G)$, then $n(a,w) < n(w,u)$.
    \item For any two adjacent vertices $a$ and $w$, we have $wt(a)$ = $wt(w) + n(w, a)- n(a,w)$.
\end{enumerate}
\end{Proposition}
\begin{proof} (1) By the assumption that $aw, wu \in E(G)$ but $au \not\in E(G)$, we obtain, if $b \in A(a,w)$ then $b  \in A(w,u)$. So $A(a,w) \subseteq A(w,u)$. Moreover, as $w \in A(w, u)  \setminus A(a,w)$, $|A(a,w)| < |A(w,u)|$.  See $a, w, u$ in  Figure 1 as an example. (2) is true  since $wt(a)-wt(w)=n(w,a) - n(a,w)$. See $a, w$ in Figure 1 as an example. 
\end{proof}
    
\begin{Lemma}
\label{lem:wg} 
Let $G$ be a block graph. Then $W(G)$ is contained in a block of $G$. 
\end{Lemma}
\begin{proof} It is enough to consider that $|W(G)| \geq 2$. Since $G$ is a block graph, to prove the result, it is enough to show  that $W(G)$ is a clique. Assuming to the contrary, there exist $x, y \in W(G)$ such that $x$ and $y$ are not adjacent.  Then $x$ and $y$ are in different blocks. Let $P_{xy}$ be the unique $(x,y)$-geodesic in $G$, $P_{xy}:  x = v_0, v_1, \ldots, v_t = y$, for some $t \geq 2$. Note that  $v_1, v_2, \ldots, v_{t-1}$ are all cut vertices, and $v_{i} v_{j} \not \in E(G)$ for all $|i-j| \geq 2$.   

By \cref{temp} (1), $n(x,v_1) < n(v_1, v_2) < \ldots < n(v_{t-1},v_{t})$, so $n(x,v_1) < n(v_{t-1}, y)$. Similarly, we have $n(y,v_{t-1}) < n(v_1,x)$. 
As $x, y \in W(G)$, by \cref{temp} (2),   
$wt(v_1)-wt(x) = n(x, v_1) - n(v_1, x) \geq 0$ and 
$wt(v_{t-1})-wt(y) = n(y, v_{t-1}) - n(v_{t-1}, y) \geq 0$, implying that $n(x, v_1) \geq n(v_1,x)$ and $n(y, v_{t-1}) \geq n(v_{t-1}, y)$. Combining the above results, we arrive at the following contradiction $$
n(v_1,x) \leq n(x,v_1) < n(v_{t-1},y) \leq n(y,v_{t-1}) < n(v_1,x). 
$$
Hence, the result follows.
\end{proof}

Let $G$ be a block graph. If $|W(G)| \geq 2$, by \cref{lem:wg}, $W(G) \subseteq B$ for some block $B$ of $G$. We call $B$ the {\it central block} of $G$. A vertex $w \in V(G)$ is a {\it central vertex} if $W(G)=\{w\}$ or $w \in V(B)$. Let $w$ be a central vertex of $G$. If a vertex $u$ is on the $(w,v)$-geodesic, then $u$ is an \emph{ancestor} of $v$, and $v$ is a \emph{descendent} of $u$. If $v$ is a descendent of $u$ and is adjacent to $u$, then $v$ is a \emph{child} of $u$, and $u$ is a {\it parent} of $v$.

Let $D$ be a non-central block that contains a central vertex $w$, then we say $D$ is {\it adjacent} to $w$. The subgraph induced by the vertices in $V(D) \setminus \{w\}$ along with their descendants is called a \emph{branch} of $G$. Two branches are  \emph{different} if they are induced by vertices of two different blocks adjacent to the same central vertex, and are called \emph{opposite} if they are induced by vertices of two different blocks adjacent to different central vertices. Accordingly, two vertices $u$ and $v$ are said to be in  \emph{different branches} (or \emph{opposite branches}, respectively) if they belong to different branches (or \emph{opposite branches} if they belong to opposite branches). Note that opposite branches only occur when $G$ contains more than one central vertex (e.g. Figure \ref{Fig:Block}). 

\begin{figure}[ht]
\begin{center}
  \includegraphics[width=2.5in]{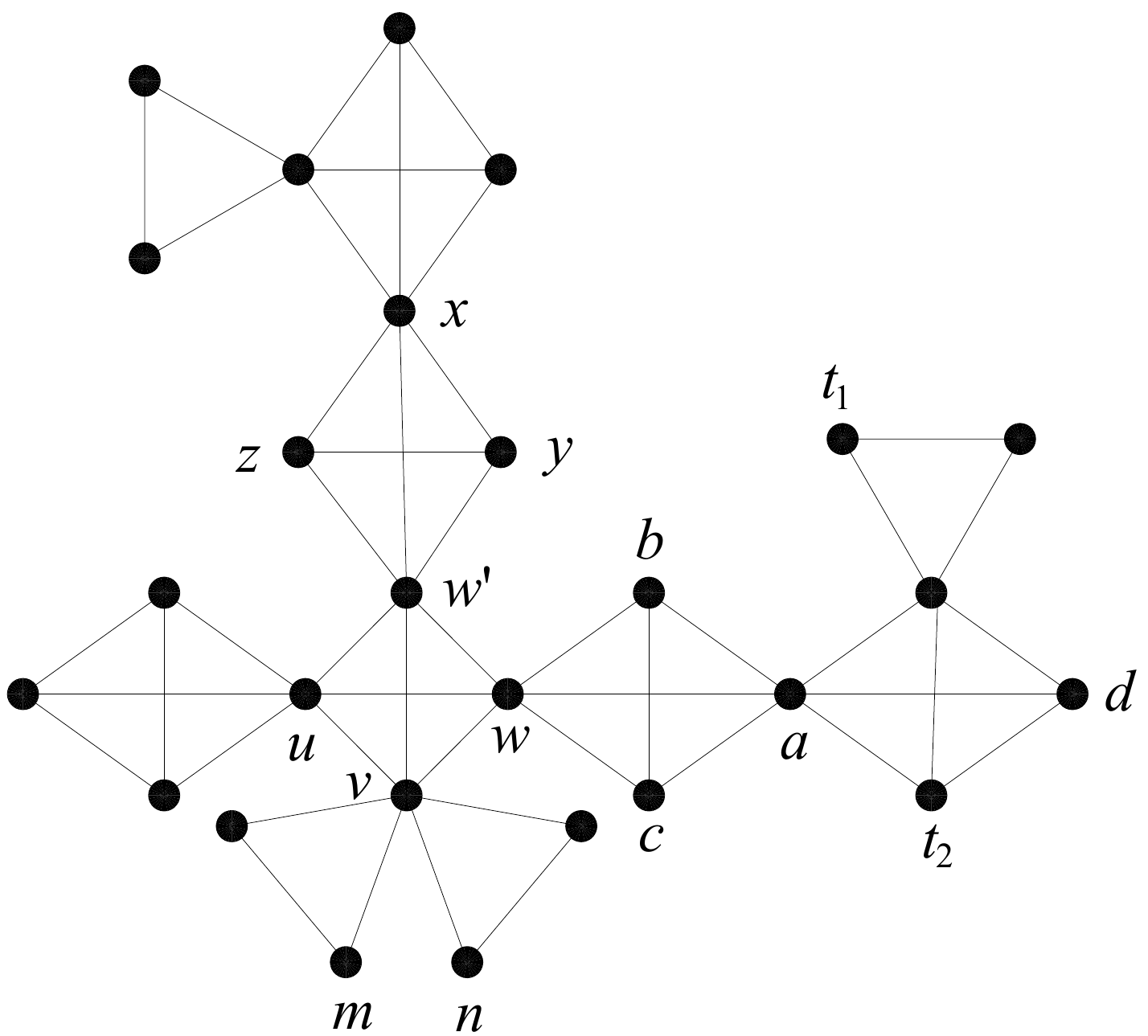}\\
  \caption{A block graph with two weight centers $W(G)=\{w, w'\}$, four central vertices $\{u, v, w, w'\}$ (which form  the central block) and  five branches (one incident to each of $w$, $w'$, and $u$, and two incident to $v$). Vertices $b$ and $y$ belong to opposite branches, while $m$ and $n$ belong to  different branches.} 
  \label{Fig:Block}
\end{center}
\end{figure}

We view a block graph ``rooted" at its central vertices.  Define the {\it level function} on $V(G)$ by $$
L_{G}(v) := \min\{d(v,w) : \mbox{$w$ is a central vertex of $G$}\}. 
$$
We simply denote $L_{G}(v)$ by $L(v)$ when $G$ is clear in the context.  The \emph{total level} of $G$ is
$$
L(G) := \displaystyle\sum_{v \in V(G)} L(v).
$$
Note that between any two vertices $u$ and $v$ in a block graph $G$, there is a unique $(u,v)$-geodesic, denoted by  $P_{uv}$. For a block graph $G$, define the following: 
\begin{center}
$\phi_{G}(u,v)$ := max\{$L_{G}(x)$ : $x$ is a common ancestor of $u$ and $v$\}.
\end{center}
\begin{equation*}
\delta_{G}(u,v) := \left\{
\begin{array}{ll}
1, & \mbox{$P_{uv}$ contains two central vertices}\\ [0.3cm]
0, & \mbox{otherwise}.
\end{array}
\right.
\end{equation*}

\begin{equation*}
\rho_{G}(u,v) := \left\{
\begin{array}{ll}
0, & \mbox{$P_{uv}$ contains a central vertex or a common ancestor of $u$ and $v$} \\ [0.3cm]
1, & \mbox{otherwise}.
\end{array}
\right.
\end{equation*}

\noindent
We denote $\phi_G(u,v)$, $\delta_G(u,v)$ and $\rho_G(u,v)$, by $\phi(u,v)$, $\delta(u,v)$ and $\rho(u,v)$, respectively, when $G$ is clear in the context. The following observations follow directly from the above definitions:

\begin{Observation}
\label{l:phi} 
Let $G$ be a block graph with diameter $d(G) \geq 2$. The  following hold for all $u, v \in V(G)$: 
\begin{enumerate}[\rm (1)] 
\item $0 \leq \phi(u,v) \leq \min\{L(u), L(v)\}$ and $\phi(u,v) \leq \max\{L(z)-1 : z \in V(G)\}$;  
\item $\phi(u,v) = \rho(u,v)= 0$ if and only if $u$ and $v$ are in different or opposite branches, or either of them is a central vertex; 
\item $\delta(u,v)= 0$ if $|W(G)|= 1$ or $u$ and $v$ are in the same branch.
\end{enumerate}
\end{Observation}

With \cref{l:phi}, we are able to describe the distance between two vertices in a block graph.  

\begin{Proposition}
\label{d:uv}
Let $G$ be a block graph, and $u, v \in V(G)$. Then $$
d(u,v) = L(u) + L(v) + \delta(u,v) - 2\phi(u,v) - \rho(u,v).
$$
\end{Proposition}
\begin{proof}
Assume that $u$ and $v$ are in different or opposite branches, or one of them is a central vertex. Then $d(u,v) = L(u)+L(v)$, if there is only one central vertex on  $P_{uv}$. If there are two central vertices on $P_{uv}$, then $d(u, v)=L(u) + L(v) +1$. By \cref{l:phi},  the result follows. 

Assume, otherwise, $u$ and $v$ are in the same branch and none of them is a central vertex. Then $P_{uv}$ contains no central vertices and hence $\delta(u,v)=0$. If $P_{uv}$ contains a common ancestor of $u$ and $v$, then $d(u,v) = L(u) + L(v) - 2 \phi(u,v)$, as $\rho(u,v)=0$. Otherwise, 
$d(u,v) = L(u) + L(v) - 2 \phi(u,v)-1$.  Hence the result follows. See $t_1, t_2$ on Figure 1 as an example.
\end{proof}

For a block graph $G$, we define
\begin{eqnarray*}
\ve(G) = & \left\{
\begin{array}{ll}
1, & \mbox{ if  $|W(G)|=1$} \\ [0.3cm]
0, & \mbox{ if  $|W(G)| \geq 2$}.
\end{array}
\right.
\end{eqnarray*}


\section{A tight lower bound for the radio number of block graphs}


In this section, we present the following results: 
\begin{itemize}
    \item A lower bound for the radio  number of block graphs (Theorem \ref{thm:lb}). 
    \item Three necessary and sufficient conditions to achieve this bound (Theorems \ref{thm:lb}, \ref{thm:main} and \ref{thm:main2}). 
\item Three sufficient conditions to achieve this bound (Theorem \ref{thm:suf}).
\end{itemize}

Let $f$ be a radio labeling of a graph $G$. Since $f$ is injective, without loss of generality, we may assume that $f$ induces an ordering $\vec{u}$ of $V(G)$, $\vec{u}=(u_{0}, u_{1}, u_{2}, ...,u_{|V(G)|-1})$, where 
$$
0 = f(u_{0}) < f(u_{1}) < f(u_{2}) < ... < f(u_{|V(G)|-1}) = span(f).
$$

A radio labeling $f$ with $span(f)=rn(G)$ is called an {\it optimal radio labeling}.  An ordering $\vec{u}=(u_{0}, u_{1}, u_{2}, ...,u_{|V(G)|-1})$ of $V(G)$ is called {\it optimal} if it is induced by an optimal radio labeling.

Now we present a lower bound for the radio number of block graphs. 

\begin{Theorem}
\label{thm:lb} 
Let $G$ be a block graph of order $p$ and diameter $d(G) \geq 2$. Then 
\begin{equation}
\label{eq:lb}
rn(G) \geq (p-1)(d(G)+\ve(G))-2L(G)+\ve(G).
\end{equation}
Moreover, the equality holds if and only if there exists an ordering $\vec{u}=(u_0,u_1,\ldots,u_{p-1})$ of $V(G)$ such that all the following hold: 
\begin{enumerate}[\rm (a)]
    \item $L(u_{0})+L(u_{p-1})= \ve(G)$;  
    \item for $0 \leq i \leq p-2$, $u_{i}$ and $u_{i+1}$ are in different branches when $|W(G)|$ = 1, and are in  opposite branches when $|W(G)| \geq 2$ (i.e. $\rho(u_i,u_{i+1}) = \phi(u_i,u_{i+1}) = 0$ and $\delta(u_i,u_{i+1}) = 1 - \ve(G)$);
    \item the mapping $f$ on $V(G)$ defined by $f(u_0)=0$ and $f(u_{i+1})=f(u_i) + d(G) + \ve(G) - L(u_{i+1}) -  L(u_i)$ for $0 \leq i \leq p-2$ is 
  a radio labeling for $G$.
\end{enumerate}
\end{Theorem}

\begin{proof} Suppose $f$ is a radio labeling of $G$ with $0 = f(u_{0}) < f(u_{1}) < f(u_{2}) < ... < f(u_{p-1})$.  By definition, $f(u_{i+1}) - f(u_{i}) \geq d(G)+1-d(u_{i},u_{i+1})$, for $0 \leq i \leq p-2$. Summing up these $p-1$ inequalities, we obtain 
\be
\label{eq:sumup} 
span(f) = f(u_{p-1}) \geq (p-1)(d(G)+1) - \displaystyle\sum_{i=0}^{p-2} d(u_{i},u_{i+1}). \ee

\noindent
\textsf{Case 1}: $|W(G)|=1$. Then $\delta(u_{i},u_{i+1})$ = 0 for $0 \leq i \leq p-2$. Since $G$ has only one central vertex,  $L(u_{0})+L(u_{p-1}) \geq 1$. Applying \cref{d:uv} and \cref{l:phi}
to (\ref{eq:sumup}), we obtain
\bean
span(f) 
&\geq& (p-1)(d(G)+1) - \displaystyle\sum_{i=0}^{p-2} \left[L(u_{i})+L(u_{i+1})-2\phi(u_{i},u_{i+1})-\rho(u_{i},u_{i+1})\right] \\
&\geq& (p-1)(d(G)+1) - 2  
L(G) + L(u_{0})+L(u_{p-1}) \\
&\geq& (p-1)(d(G)+1)  - 2 L(G) + 1.
\eean

\noindent
\textsf{Case 2}: $|W(G)| \geq 2$. 
Since $L(u_{0}), L(u_{p-1}) \geq 0$, applying \cref{d:uv} and  \cref{l:phi} to (\ref{eq:sumup}), we obtain
\bean
span(f) 
&\geq& (p-1)(d(G)+1) - \displaystyle\sum_{i=0}^{p-2} \left[L(u_{i})+L(u_{i+1})+\delta(u_{i},u_{i+1})-2\phi(u_{i},u_{i+1})-\rho(u_{i},u_{i+1})\right]  \\ 
&\geq& (p-1)(d(G)+1) - \displaystyle\sum_{i=0}^{p-2} \left[L(u_{i})+L(u_{i+1})+1\right] \\
&=& (p-1)(d(G)+1) - 2 L(G) + L(u_{0})+L(u_{p-1})-(p-1) 
\geq (p-1)d(G) - 2 L(G). 
\eean
Hence, the lower bound is true.  In addition, it is easy to see in  Cases 1 and 2, all the equalities hold if 
and only if the moreover part is true. 
\end{proof}

For a block graph $G$, we denote the lower bound of $rn(G)$ in \cref{thm:lb} by $LB(G)$.  If  $rn(G)=LB(G)$ then we call $G$ a {\it lower bound block graph}. If $T$ is a tree and $rn(T)=LB(T)$, then $T$ is also called a {\it lower bound tree} (cf. \cite{CLS, ImprovedBound}). 

In the next two results, we prove additional characterizations (besides the moreover part in \cref{thm:lb}) of lower bound block graphs.

\begin{Theorem}\label{thm:main} Let $G$ be a block graph of order $p$ and diameter $d(G) \geq 2$. Then
$G$ is a lower bound block graph if and only if there exists an ordering $\vec{u}=(u_{0},  u_{1}, \ldots, u_{p-1})$ of $V(G)$ such that \cref{thm:lb} (a) holds and the following is true for all $0 \leq i < j \leq p-1$:
\begin{equation}\label{eq:dij} 
d(u_{i},u_{j}) \geq \displaystyle\sum_{t=i}^{j-1} [L(u_{t})+L(u_{t+1})] - (j-i)(d(G)+\ve(G)) + d(G) +1. 
\end{equation}
\end{Theorem}
\begin{proof}~
\textsf{Necessity}:~Suppose that $rn(G) = LB(G)$. By Theorem \ref{thm:lb}, there exists an optimal radio labeling $f$ with ordering $\vec{u}=(u_{0},  u_{1}, \ldots, u_{p-1})$ of $V(G)$ such that Theorem \ref{thm:lb} (a), (b) and (c) hold. 
By Theorem \ref{thm:lb} (c), 
$$
f(u_{j}) - f(u_{i}) = \displaystyle\sum_{t=i}^{j-1} \left[f(u_{t+1}) - f(u_{t})\right] =  \displaystyle\sum_{t=i}^{j-1} \left[d(G)+\ve(G)-L(u_{t})-L(u_{t+1})\right] \geq  d(G)+1-d(u_{i},u_{j}). 
$$
Hence, \eqref{eq:dij} is true. 

\textsf{Sufficiency}:~Suppose there exists an ordering  $\vec{u}=(u_0,u_1,\ldots,u_{p-1})$ of $V(G)$ so that (a) in   \cref{thm:lb} and \eqref{eq:dij} are satisfied for all $u_i$ and $u_j$, $0 \leq i < j \leq p-1$. We now prove that conditions (b) and (c) in Theorem \ref{thm:lb} are satisfied. Take $j=i+1$ in \eqref{eq:dij} we get $d(u_i,u_{i+1}) \geq L(u_i)+L(u_{i+1})+(1-\ve(G))$, which implies  that $u_i$ and $u_{i+1}$ are in different branches when $|W(G)|=1$, and in opposite branches when $|W(G)| \geq 2$. Hence \cref{thm:lb} (b) holds. 

It remains to prove that $f$ defined in \cref{thm:lb} (c) is a radio labeling. Let $u_{i}$ and $u_{j}$ ($0 \leq i < j \leq p-1$) be two arbitrary vertices. By \eqref{eq:dij}, 
$$
f(u_{j}) - f(u_{i}) = (j-i)(d(G)+\ve(G))-\displaystyle\sum_{t=i}^{j-1}(L(u_{t})+L(u_{t+1})) 
\geq d(G) + 1 - d(u_{i},u_{j}). 
$$
Hence $f$ is a radio labeling. The proof is complete.
\end{proof}

In \cite[Theorem 3]{Daphne1} a lower bound of the radio number of trees was given. The above two results generalize \cite[Lemma 3.1]{Bantva2} and \cite[Theorem 3.2]{Bantva2}. Precisely, when $G$ is a tree then Theorem \ref{thm:lb} and Theorem \ref{thm:main} are the same as  \cite[Lemma 3.1]{Bantva2} and \cite[Theorem 3.2]{Bantva2}, respectively. 
\begin{Theorem}\label{thm:main2} Let $G$ be a block graph of order $p$ and diameter $d(G) \geq 2$. Then $G$ is a lower bound block graph 
if and only if there exists an ordering $\vec{u}=(u_{0},u_{1},...,u_{p-1})$ of $V(G)$ such that (a) and (b) in \cref{thm:lb} hold, and the following are  true: 
\begin{enumerate}[\rm (a*)]
\item for all $i$, $L(u_{i}) \leq \frac{d(G)+\ve(G)}{2}$; 
\item any two vertices $u_{i}$ and $u_{j}$ $(0 \leq i < j \leq p-1)$ in the same branch satisfy 
\be\label{eq:pij} 
\phi(u_{i},u_{j}) \leq (j-i-1)\left(\frac{d(G)+\ve(G)}{2}\right)- \left(\displaystyle\sum_{t=i+1}^{j-1}L(u_{t}) \right) -\left(\frac{1- \ve(G)}{2}\right).  \ee
\end{enumerate}
\end{Theorem}
\begin{proof}~\textsf{Necessity}:~Suppose $rn(G) = LB(G)$.  
By Theorem \ref{thm:main}, there exists an ordering $\vec{u}=(u_{0},u_{1},...,u_{p-1})$ of $V(G)$ such that \eqref{eq:dij} in \cref{thm:main} holds for any two vertices $u_i$ and $u_j, 0 \leq i < j \leq p-1$. Applying \eqref{eq:dij} with $d(u_{i-1},u_{i+1})$ and by \cref{d:uv}, we obtain 
$$
2L(u_{i}) \leq d(G) + 2 \ve(G) + \delta(u_{i-1},u_{i+1}) - 1.  
$$
Note that $\ve(G)=1$ when $|W(G)|=1$, and $\rho(u_{i-1}, u_{i+1})=1$ is true only when $|W(G)| \geq 2$.  Thus (a*) is satisfied. 

To prove (b*), assume $u_{i}$ and $u_{j}$ are in the same branch. Then $\delta(u_i, u_j)=0$. Combining \cref{d:uv} (for $d(u_i, u_j)$) with  \eqref{eq:dij} in \cref{thm:main}, as $\rho(u_i,u_j) \geq 0$, we obtain \eqref{eq:pij}.  Hence (b*) is true.

\textsf{Sufficiency}:~Suppose there exists an ordering $\vec{u} = (u_{0},u_{1},...,u_{p-1})$ of $V(G)$ such that (a) and (b) in \cref{thm:lb} as well as (a*) and (b*) of the statement hold. It is enough to prove that  \eqref{eq:dij} in \cref{thm:main} holds for any $u_i$ and $u_j$, $0 \leq i < j \leq p-1$. 

If $u_i$ and  $u_j$ are in different or opposite branches, then $d(u_{i},u_{j}) = L(u_{i})+L(u_{j})+ 1 - \ve(G)$. By (a*), the right-hand-side of \eqref{eq:dij} is at most $d(G)+1$. Hence \eqref{eq:dij} holds. 
Assume $u_{i}$ and $u_{j}$ are in the same branch. Then $j \geq i+2$. 
By direct calculation, \eqref{eq:dij} can be obtained by  \cref{d:uv} and \eqref{eq:pij}. Hence the proof is complete. 
\end{proof}

Now we turn to the last part of this section. Namely, we present three sufficient conditions for a block graph to be a lower bound block graph (i.e. the lower bound in \cref{thm:lb} is achieved). 

\begin{Theorem}\label{thm:suf} Let $G$ be a block graph of order $p$ and diameter $d(G) \geq 2$. Then $G$ is a lower bound block graph if   
there exists an ordering $\vec{u}=(u_{0},u_{1},\ldots,u_{p-1})$ of $V(G)$ such that (a) and (b) in \cref{thm:lb} hold, and  
one of the following is true:
\begin{enumerate}[\rm (i)]
  \item $\min\{d(u_{i},u_{i+1}),d(u_{i+1},u_{i+2})\} \leq (d(G)+1-\ve(G))/2$, for all $0 \leq i \leq p-3$;
  \item $d(u_{i},u_{i+1}) \leq (d(G)+1+\ve(G))/2$, for all $0 \leq i \leq p-2$;
  \item For all $i$, $L(u_{i}) \leq (d(G)+\ve(G))/2$; and whenever $u_{i}$ and $u_{j}$ are in the same branch then $j-i \geq d(G)$.
\end{enumerate}
\end{Theorem}
\begin{proof} Assume (a) and (b) in  \cref{thm:lb} are true.  We prove if one of  (i), (ii), (iii) holds, then the ordering $\vec{u}=(u_{0}, u_{1},\ldots, u_{p-1})$ of $V(G)$ satisfies (a*) and (b*) of Theorem  \ref{thm:main2}.  Denote the right-hand-side of \eqref{eq:pij} by 
$
P_{i,j}.
$
To show \cref{thm:main2} (b*) we verify that  $\phi(u_i, u_j) \leq P_{i,j}$ for any vertices $u_i$ and $u_j$ from the same  branch.  We prove this by considering two cases. 

\noindent
\textsf{Case 1.} \ $|W(G)|$ = 1. Then   $\ve = 1$. 

\textsf{Sub-Case 1.1} 
Suppose (i) holds.  By \cref{thm:lb} (b), for each $1 \leq i \leq p-2$, $L(u_{i}) \leq \min \{L(u_{i})+L(u_{i+1}),  L(u_{i})+L(u_{i-1})\} = \min\{ d(u_{i},u_{i+1}),d(u_{i},u_{i-1})\} \leq d(G)/2$. 
Thus, \cref{thm:main2} (a*) is true.  Furthermore, by \cref{thm:lb} (a),  $L(u_0), L(u_{p-1}) \leq 1$. Hence $L(u_{i})+L(u_{i+1}) + L(u_{i+2}) \leq d(G)$ for all $0 \leq i \leq p-3$.      

Let $u_{i}$ and $u_{j}$ ($0 \leq i < j \leq p-1$) be vertices from the same branch of $G$. Denote $j-i-1=3q+r$ for some $q \geq 1$ and $0 \leq r \leq 2$. Then 
$$
P_{i,j} \geq (3q+r)(d(G)+1)/2 - qd(G) - rd(G)/2 \geq  (d(G)+3)q/2 \geq d(G)/2 \geq \phi(u_i, u_j). $$
It remains to show the cases when $j-i-1=1$ and $j-i-1=2$. Assume $j-i=2$. Then $L(u_i) + L(u_{i+1}) \leq d(G)/2$ or $L(u_{i+1}) + L(u_{i+2}) \leq d(G)/2$. Assume $L(u_i) + L(u_{i+1}) \leq d(G)/2$. (The case for $L(u_{i+1}) + L(u_{i+2}) \leq d(G)/2$ can be proved similarly.) Then $L(u_{i+1}) \leq d(G)/2 - L(u_i)$. Hence 
$$
P_{i,j} \geq (d(G)+1)/2 - d(G)/2 + L(u_i) \geq  1/2 + L(u_i) \geq \phi(u_i, u_j). $$
Finally, assume $j-i=3$.  If  $L(u_{i+1}) + L(u_{i+2}) \leq d(G)/2$,  
then one can see that $P_{i,j} \geq d(G)/2$  and the result follows. Thus, assume $L(u_{i+1}) + L(u_{i+2}) \geq d(G)/2$. By the assumption of (i), this implies $L(u_{i}) + L(u_{i+1}) \leq d(G)/2$ and $L(u_{i+2}) + L(u_{i+3}) \leq d(G)/2$. So, $L(u_{i+1}) \leq d(G)/2 - L(u_{i})$ and $L(u_{i+2}) \leq d(G)/2 - L(u_{i+3})$.  Hence 
$$
P_{i,j} \geq d(G)+1 - \left[d(G)/2 -  L(u_i)\right] - \left[d(G)/2 - L(u_j)\right] = 1  + L(u_i) + L(u_j) \geq \phi(u_i, u_j).   
$$
Therefore, \cref{thm:main2} (b*) holds. 

\textsf{Sub-Case 1.2.}  Suppose (ii) holds. Then by \cref{thm:lb} (a)  and (b),  $d(u_{i},u_{i+1}) = L(u_i)+L(u_{i+1}) \leq (d(G)+2)/2$ for each $0 \leq i \leq p-2$ and $L(u_{0})$ = 0.
Thus  $L(u_{i}) \geq 1$ for $1 \leq i \leq p-1$.  So we have $L(u_{i}) \leq (d(G)+1)/2$ for $0 \leq i \leq p-1$. Hence \cref{thm:main2} (a*) holds. 

Let $u_{i}$ and $u_{j}$ be two vertices of $G$ from  the same branch. 
If $j = i+2$, then 
$ 
P_{i,j} \geq (d(G)+1)/2 - L(u_{i+1}) \geq L(u_{i}) - 1/2$. Hence $P_{i,j} \geq \phi(u_{i},u_{j})$.

Assume $j \geq i+3$. Then 
$$
P_{i,j} \geq (d(G)+1)-L(u_{i+1})-L(u_{i+2}) \geq d(G)/2 \geq \phi(u_{i},u_{j}).
$$

\textsf{Sub-Case 1.3.}  
Suppose (iii) holds.  Assume $d(G)$ is even. Then for every $i$, $L(u_i) \leq d(G)/2$. By (iii), 
$P_{i,j} \geq (j-i-1)((d(G)+1)/2)-(j-i-1)(d(G)/2) = (j-i-1)/2 \geq (d(G)-1)/2 \geq \phi(u_i,u_j)$.

Assume $d(G)$ is odd.  
Because $L(u_{t}) \leq (d(G)+1)/2$ and  $\max\{L(u_{t})+L(u_{t+1}):0 \leq t \leq p-2\} \leq d(G)$, we obtain  
$\sum_{t=i+1}^{j-1}L(u_{t}) \leq ((j-i-1)/2)((d(G)+1)/2)+((j-i-1)/2)((d(G)-1)/2)$. Hence, by (iii), $P_{i,j} \geq (j-i-1)/2 \geq (d(G)-1)/2 \geq \phi(u_{i},u_{j})$. 

\noindent
\textsf{Case 2}:~$|W(G)| \geq 2$. Then   $\ve = 0$. Note that in this case, (ii) implies (i). Thus, we only need to consider (i) and (iii).  

\textsf{Sub-Case 2.1.} Suppose (i) holds. By  \cref{thm:lb} (a), $L(u_0)=L(u_{p-1})=0$. 
By (i) and \cref{thm:lb} (b), for each $1 \leq i \leq p-2$, $L(u_i) \leq  \min\{L(u_i)+L(u_{i+1}), L(u_i)+L(u_{i-1})\} =  \min\{d(u_i,u_{i+1})-1,d(u_i,u_{i-1})-1\} \leq (d(G)-1)/2 < d(G)/2$. Thus, \cref{thm:main2} (a*) is true. Moreover, note that $L(u_i)+L(u_{i+1})+L(u_{i+2}) \leq d(G)-1$.

Let $u_i$ and $u_j\;(0 \leq i < j \leq p-1)$ be vertices from the same branch of $G$. Denote $j-i-1 = 3q+r$ for some $q \geq 1$ and $0 \leq r \leq 2$. Then $$
P_{i,j} \geq (3q+r)(d(G)/2)-q(d(G)-1)-r((d(G)-1)/2)-1/2 \geq (qd(G)+1)/2 \geq d(G)/2 \geq \phi(u_i,u_j).
$$

It remain to show the cases when $j-i-1=1$ and $j-i-1=2$. Assume $j-i=2$. Then $L(u_i)+L(u_{i+1}) \leq (d(G)-1)/2$ or $L(u_{i+1})+L(u_{i+2}) \leq (d(G)-1)/2$. Assume $L(u_i)+L(u_{i+1}) \leq (d(G)-1)/2$ (The case $L(u_{i+1})+L(u_{i+2}) \leq (d(G)-1)/2$ can be proved similarly). Then $L(u_{i+1}) \leq (d(G)-1)/2-L(u_i)$. Hence,
$$
P_{i,j} \geq d(G)/2-(d(G)-1)/2+L(u_i)-1/2 = L(u_i) \geq \phi(u_i,u_j).
$$

Finally, assume $j-i=3$. If $L(u_{i+1})+L(u_{i+2}) \leq (d(G)-1)/2$, then one can see that $P_{i,j} \geq d(G)/2$ and the result follows: Thus, assume $L(u_{i+1})+L(u_{i+2}) \geq d(G)/2$. By the assumption of (i), this implies that $L(u_i)+L(u_{i+1}) \leq (d(G)-1)/2$ and $L(u_{i+2})+L(u_{i+3}) \leq (d(G)-1)/2$. So, $L(u_{i+1}) \leq (d(G)-1)/2-L(u_i)$ and $L(u_{i+2}) \leq (d(G)-1)/2-L(u_{i+3})$. Hence,
$$
P_{i,j} \geq d(G)-[(d(G)-1)/2-L(u_i)]-[(d(G)-1)/2-L(u_j)]-1/2 = 1/2+L(u_i)+L(u_j) \geq \phi(u_i,u_j).
$$

Therefor, \cref{thm:main2} (b*) holds.

\textsf{Sub-Case 2.2.} Suppose (iii) holds. Assume $d(G)$ is even. Then $L(u_t) \leq d(G)/2$ and  $\max\{L(u_t)+L(u_{t+1}) : 0 \leq t \leq p-2\} \leq d(G)-1$. These imply that 
$\sum_{t=i+1}^{j-1}L(u_t) \leq ((j-i-1)(d(G)-1))/2$. Hence by (iii), $P_{i,j} \geq (j-i-2)/2 \geq (d(G)-2)/2 \geq \phi(u_i,u_j)$.

Assume $d(G)$ is odd. Then $L(u_t) \leq (d(G)-1)/2$ for any $0 \leq t \leq p-1$. Again, by (iii), $P_{i,j} \geq (j-i-2)/2 \geq (d(G)-2)/2 \geq \phi(u_i,u_j)$.  The proof is complete. 
\end{proof}

%
\section{Lower bound block graphs}
%

Using the results in Section 3 we present two families of lower bound block graphs. Namely, the level-wise regular block graphs and the extended star of blocks.

\subsection{Level-wise regular block graphs}

We call a non-cut-vertex in a block graph an {\it end vertex}. Let $r, m$ be positive integers, $m \neq 2$, and let $(k_{i}, m_{i})$, $1 \leq i \leq r$, be pairs of positive integers with $m_i \geq 2$. If $m=1$, we additionally assume $k_1 \geq 2$. The {\it level-wise regular block graphs}, denoted by $G^m_{(k_1,m_1) (k_2,m_2) \ldots (k_r,m_{r})}$, is defined by the following: 

\begin{itemize} 
\item  {\bf Initially ($i=0, 1$):} For $i=0$, take a clique $K_{m}$ with vertex set $ \{w^0, w^1, \ldots, w^{m-1}\}$; for $i=1$, identify one vertex of each of $k_{1}$ copies of $K_{m_{1}+1}$ to each of $w^0, w^1, w^2, \ldots, w^{m-1}$.
\item {\bf Inductively:} At the $i$-th step, $2 \leq i \leq r$, identify one vertex of each of  $k_{i}$ copies of clique $K_{m_{i}+1}$ to each end vertex of the cliques joined in the $(i-1)$-th step.
\end{itemize}
See Figure 2 for examples. 
Denote 
$G^{m}_{(k_{1},m_{1}) \cdots (k_{r},m_{r})}$ by $G^{m}$, $m \neq 2$. It is easy to see that 
$W(G^m)=\{w^0, w^1, \ldots, w^{m-1}\}$. 
Moreover, $\diam(G^{1})$ = $2r$, and $\diam(G^{m})$ = $2r+1$ if $m \geq 3$.

\begin{figure}[ht]
\begin{center}
  \includegraphics[width=6.2in]{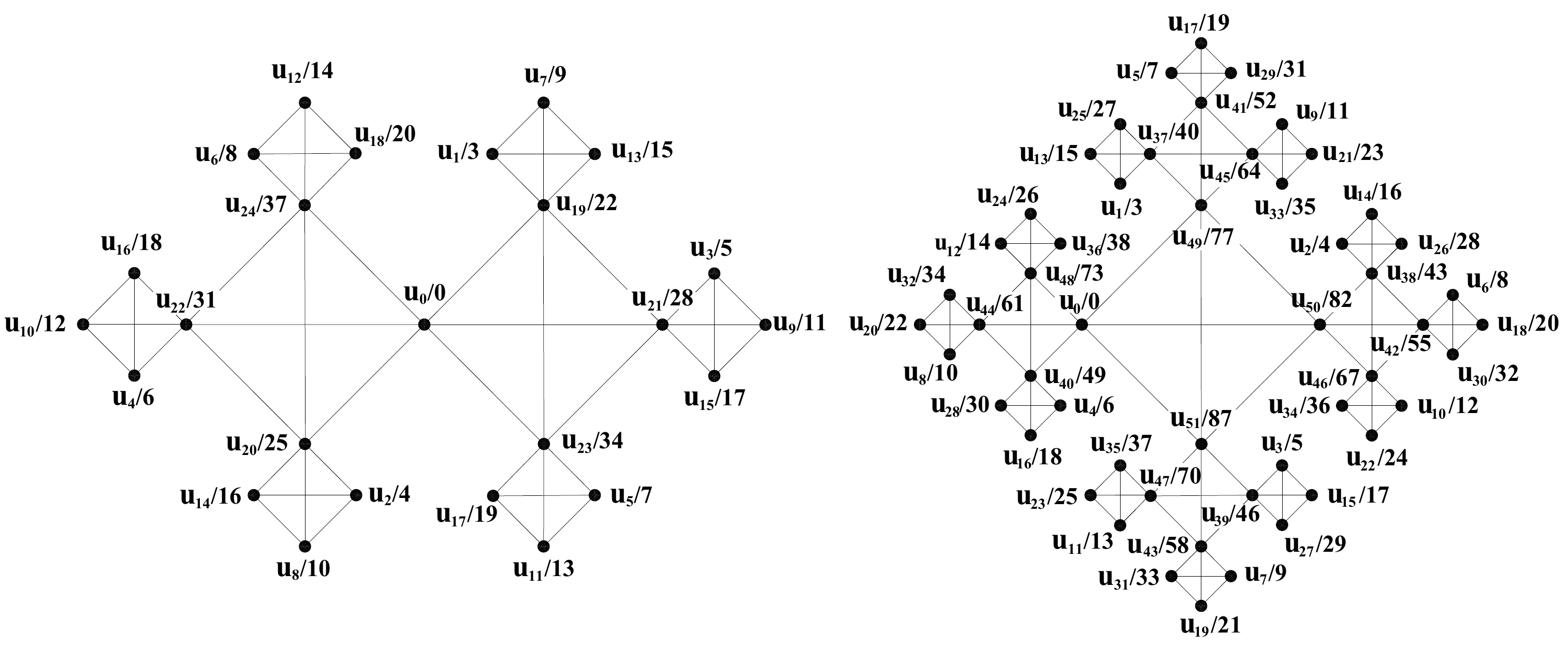}\\
  \caption{Optimal radio labelings for $G_{(2,3)(1,3)}^{1}$ and $G_{(1,3)(1,3)}^{4}$, respectively.}\label{Fig:Symblock}
\end{center}
\end{figure}

\begin{Theorem}\label{thm:level} 
Let $G = G^{m}_{(k_{1},m_{1}) \cdots (k_{r},m_{r})}$. Then $G$ is a lower bound block graph and  
\small{\begin{equation}\label{rn:level}
rn(G) = \left[  (m-1)+m\displaystyle\sum_{i = 1}^{r} \(\displaystyle\prod_{1 \leq j \leq i} k_{j}m_{j}\) 
\right] 
(2r+1)-2m\displaystyle\sum_{i = 1}^{r} i\(\displaystyle\prod_{1 \leq j \leq i} k_{j}m_{j}\) + \ve(G).
\end{equation}}
\end{Theorem}
\begin{proof} The order and the total level of $G$ are 

$$
|V(G)| = m \left[ 1+ \displaystyle\sum_{i = 1}^{r} \(\displaystyle\prod_{1 \leq j \leq i} k_{j}m_{j}\)\right], \ \ 
L(G) = m\displaystyle\sum_{i = 1}^{r} i\(\displaystyle\prod_{1 \leq j \leq i} k_{j}m_{j}\).  
$$

Substituting the above into \eqref{eq:lb} we obtain the right-hand-side of \eqref{rn:level} as a lower bound for $rn(G)$. We now prove that this lower bound is tight by giving an ordering $\vec{u} = (u_{0},u_{1},...,u_{p-1})$ of $V(G)$ which satisfies \cref{thm:main}. 

Denote  $W(G)=\{w^{0},w^{1},...,w^{m-1}\}$. In the following we denote each vertex $v$ with $L(v)=x$, $1 \leq x \leq r$, with an $x$-bit sub-index. 

\noindent
\underline{Initially}: Assign a 1-bit sub-index to all level-1 vertices. For each $0 \leq t \leq m-1$, denote the $k_{1}m_{1}$ children of $w^{t}$ by  $w^{t}_{0},w^{t}_{1},...,w^{t}_{k_{1}m_{1}-1}$ such that any $k_{1}$ consecutive vertices (i.e., vertices with consecutive sub-indices) are in different blocks while $w^{t}_{a}$ and $w^{t}_{a+k_{1}}$ are in the same block for $0 \leq a \leq k_{1}m_{1}-k_{1}-1$. 

\noindent
\underline{Inductively}: Suppose all vertices up to level-$l$ are assigned with a sub-index. Denote the $k_{l+1}m_{l+1}$ children of  $w^{t}_{i_{1},i_{2},...,i_{l}}$, $0 \leq t \leq m-1$, $0 \leq i_{j} \leq k_{j}m_{j}-1$, $0 \leq j \leq l$,  
by  $w^{t}_{i_{1},i_{2},...,i_{l},i_{l+1}}$, where $0 \leq i_{l+1} \leq k_{l+1}m_{l+1}-1$ such that any $k_{l+1}$ consecutive vertices are in different blocks while  $w^{t}_{i_{1},i_{2},...,i_{l},a}$ and  $w^{t}_{i_{1},i_{2},...,i_{l},a+k_{l+1}}$ are in the same block for $0 \leq a \leq k_{l+1}m_{l+1}-k_{l+1}-1$. Continue this process until all vertices are indexed. 

Next, we define an ordering  $\vec{u} =  (u_0,u_1,\ldots,u_{p-1})$ as follows: Let $u_0 = w^{m-1}$, and for $1 \leq j \leq p-m$, let 
$ 
u_{j}:=w^{t}_{i_{1},i_{2},...,i_{l}}$, where 
$$
j = m\left(i_{1}+i_{2}(k_{1}m_{1})+...+i_{l}\(\displaystyle\prod_{i=1}^{l-1}k_{i}m_{i}\)+\displaystyle\sum_{l+1 \leq t \leq  r}\(\displaystyle\prod_{i=1}^{t}k_{i}m_{i}\)\right)+t+1.
$$
For $p-m+1 \leq j \leq p-1$, let 
$ 
u_j:=w^{j-p+m-1}. 
$

Note that when $|W(G)|=1$, we have $u_0 \in W(G)$ and $u_{p-1}$ is adjacent to $u_0$; and  when $|W(G)| \geq 2$, we have $\{u_0, u_{p-1}\} \subseteq W(G)$. Hence \cref{thm:lb} (a) holds.  Moreover, for all $i$, $u_i$ and $u_{i+1}$ are in different branches when $|W(G)|=1$, and are in opposite branches when $|W(G)| \geq 2$. Hence \cref{thm:lb} (b) holds. To complete the proof it remains to show that the ordering $\vec{u}$ defined above satisfies \eqref{eq:dij} in \cref{thm:main}.

For $0 \leq x \leq r$, denote $L_{x}$ the set of level-$x$ vertices.  Then $|L_{x}|$ = $m\prod_{i=1}^{x} k_{i}m_{i}$. Consider any two vertices $u_{i}$, $u_{j}$ with $0 \leq i < j \leq p-1$ such that $u_{i} \in L_{a}$ and $u_{j} \in L_{b}$. Observe that $a \geq b$ as $i < j$. The right-hand-side of $(\ref{eq:dij})$ is
$$
\begin{array}{llll}
S_{i,j}&:=& \sum\limits_{t=i}^{j-1} \left[   L(u_{t})+L(u_{t+1})\right] -  (j-i)\left[d(G)+\ve(G)\right] +1 + d(G) \\
&\leq& 2a(j-i-1)+a+b-(j-i-1)\left[d(G)+\ve(G)\right]+1-\ve(G) \\
&=& a+b+1-\ve(G)-(j-i-1)\left[d(G)+\ve(G)-2a\right].
\end{array}
$$
Note that $1 - \ve(G) \geq 0$, $j-i-1 \geq 0$,  and $d(G) +\ve(G) - 2a \geq 1$.  If $u_{i}$ and $u_{j}$ are in opposite branches, then $S_{i,j} \leq a+b+1$ = $d(u_{i},u_{j})$, as $\ve(G)=0$. If $u_{i}$ and $u_{j}$ are in different  branches, then $S_{i,j} \leq a+b$ = $d(u_{i},u_{j})$. 

Assume $u_{i}$ and $u_{j}$ are in the same branch. By definition of the ordering, $j-i = \alpha m$ (if $m \geq 3$) or $j-i = \alpha k_1$ (if $m=1$) for some   
$\alpha > \phi(u_i, u_j)$. 
Hence $d(u_i, u_j) \geq a+b+1-2\alpha$.  
If $|W(G)|=1$ then $S_{i,j} \leq  a+b-(k_1\alpha-1)(d(G)+1-2a) \leq a+b+1-2\alpha \leq d(u_i,u_j)$. If $|W(G)| = m \geq 3$, then $S_{i,j} \leq a+b+1-(m\alpha-1)(d(G)-2a) \leq a+b+1-2\alpha \leq d(u_i,u_j)$. This completes the proof.
\end{proof}

The readers are referred to Figure \ref{Fig:Symblock} about the ordering used in the proof of Theorem \ref{thm:level}.

\subsection{Extended star of blocks}

Let $h, n$ be positive integers, and $P_{h+1}$ be the path with $h$ edges. The {\it path of cliques}, denoted by $P_{h, n}$, is obtained by the following process: 

\begin{itemize}
    \item For each edge $e=uv$ on $P_{h+1}$, add  $n-2$ vertices,  $w_{e,1}, w_{e,2}, \ldots, w_{e,n-2}$, and edges so that $u,v, w_{e,1}, w_{e,2}, \ldots, w_{e,n-2}$ form a clique. A vertex $v$ in $P_{n,k}$ is called a {\it tail vertex} if its eccentricity (the maximum distance from a vertex to $v$) is equal to the diameter of $P_{h,n}$. 
\end{itemize}
The extended star of blocks is defined by the following:  Start with a clique $K_m$ with vertex set $\{w^1, w^2, \ldots, w^m\}$. Let $k$ be a positive integer. In the case $m=1$, we additionally assume $k \geq 3$.  The {\it extended star of blocks}, denoted by  $S^{m}_{k, h, n}$, is established by identifying one tail vertex of each of the $k$ copies of $P_{h,n}$ to each  $w^i$, $1 \leq i \leq m$. 
See Figure 3 for examples. 

It is clear that $W(S^{m}_{k,h,n})= \{w^1, w^2, \ldots, w^m\}$ and  $S^{m}_{k,h,n}$ has $mk$ branches.   
Moreover, $\diam(S^{m}_{k,h,n}) = 2h$ if $m = 1$, and $\diam(S^{m}_{k,h,n}) = 2h+1$ if $m \geq 2$.

\begin{figure}[ht]
\begin{center}
  \includegraphics[width=6in]{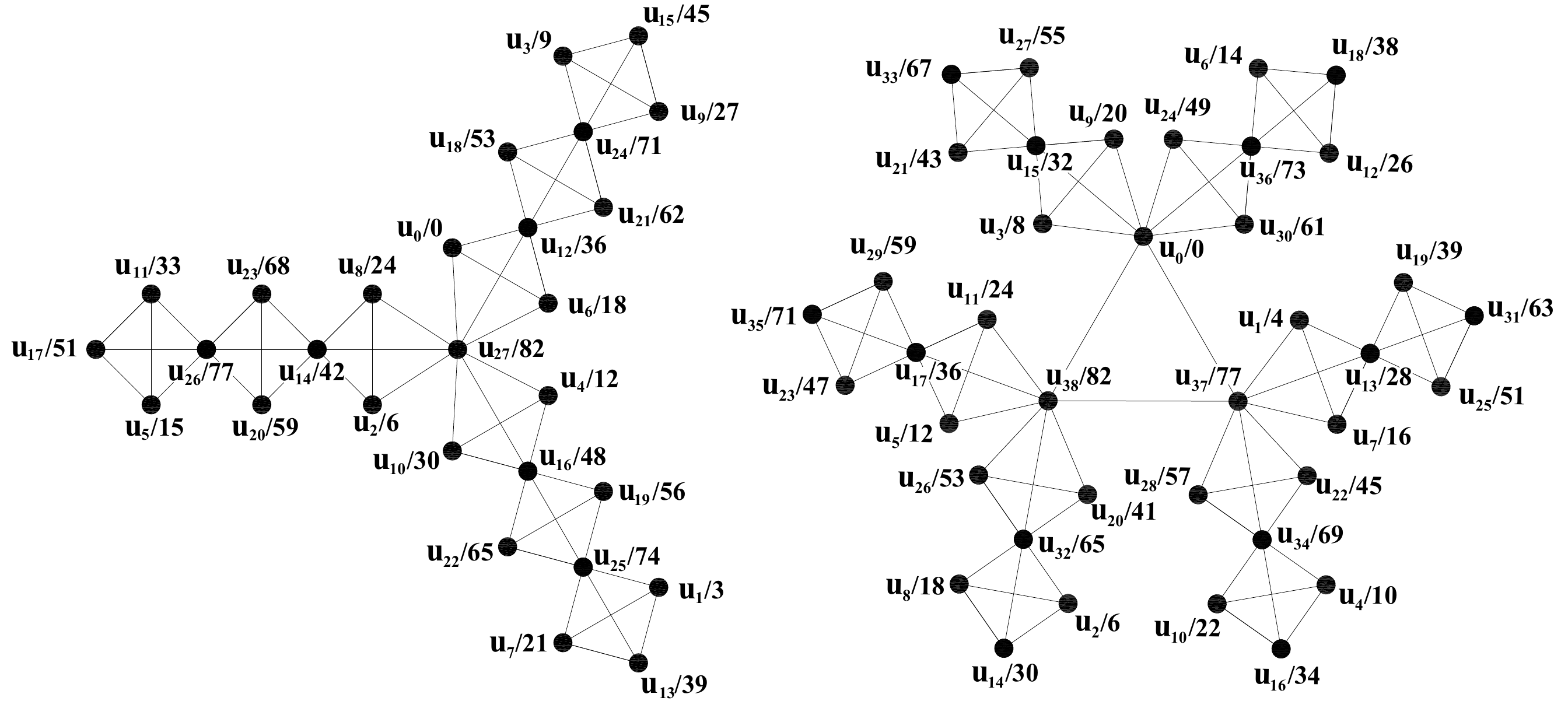}\\
  \caption{Optimal radio labelings of $S^{1}_{3,3,4}$ and $S^{3}_{2,2,4}$, respectively.}\label{Fig:Starblock}
\end{center}
\end{figure}

\begin{Theorem}\label{star:thm} Let $G = S^{m}_{k,h,n}$. Then $G$ is a lower bound block graph and 
\begin{equation}\label{star:rn}
rn(G) = mkh^2(n-1)+(m-1)(2h+1)+\ve(G).
\end{equation}
\end{Theorem}
\begin{proof} It is straightforward to verify that the order and the total level of graph $G$ are 
$$
|V(G)| = m [kh(n-1) + 1], \ \ 
L(G) = \frac{mkh(h+1)(n-1)}{2}. 
$$
Substituting the above into (\ref{eq:lb}) we obtain the right-hand-side of (\ref{star:rn}) as a lower bound for $rn(G)$. Now we prove that this lower bound is tight by giving an optimal ordering  $\vec{u}=(u_0,u_1,...,u_{p-1})$ of $V(G)$ which  will induce a labeling with the desired span. 

Denote $W(G)=\{w^1, w^2, \ldots, w^m\}$.  For each $w^i$, denote the $k$ branches adjacent  to it by $B_{i+mt}$ for $0  \leq t \leq k-1$. For $1 \leq i \leq h, 1 \leq j \leq n-1, 1 \leq l \leq mk$, denote  $w_{i,j}^{l}$, the $j$-th vertex of level-$i$ in $B_l$ where $w_{i, n-1}^{l}$ are the cut-vertices of  $G$, $1 \leq i \leq h-1, 1 \leq l \leq mk$.

We now rename $w^l_{i,j}, 1 \leq i \leq h, 1 \leq j \leq n-1, 1 \leq l \leq mk$, as \{$w_1,\ldots,w_{p-m}$\}: 

\textsf{Case A:} $km$ is odd. 
If $h$ is odd, let $w_{t} := w_{i,j}^{l}$, where \begin{eqnarray*}
t = &\left\{
\begin{array}{ll}
2mk(i-1)(n-1)+2mk(j-1)+l, & \mbox{ if } i < \lceil\frac{h}{2}\rceil, \ \mbox{$l$ is odd}  \\ [0.3cm]
2mk(i-1)(n-1)+mk(j-1)+l, & \mbox{ if } i = \lceil\frac{h}{2}\rceil, \ \mbox{$l$ is odd}  \\ [0.3cm]
2mk(h-i)(n-1)+2mk(j-1)+l+mk, & \mbox{ if } i > \lceil\frac{h}{2} \rceil, \ \mbox{$l$ is odd} \\ [0.3cm]
2mk(i-1)(n-1)+2mk(j-1)+l+mk, & \mbox{ if } i < \lceil\frac{h}{2}\rceil, \ \mbox{$l$ is even}  \\ [0.3cm]
2mk(i-1)(n-1)+mk(j-1)+l, & \mbox{ if } i = \lceil\frac{h}{2}\rceil, \ \mbox{$l$ is even}  \\ [0.3cm]
2mk(h-i)(n-1)+2mk(j-1)+l, & \mbox{ if } i > \lceil\frac{h}{2}\rceil, \ \mbox{$l$ is even}.
\end{array}
\right.
\end{eqnarray*}

If $h$ is even, let 
$w_{t} := w_{i,j}^{l}$, where  
\begin{eqnarray*}
t = &\left\{
\begin{array}{ll}
2mk(i-1)(n-1)+2mk(j-1)+l, & \mbox{ if } i \leq \frac{h}{2}, \ \mbox{$l$ is odd} \\ [0.3cm]
2mk(h-i)(n-1)+2mk(j-1)+l+mk, & \mbox{ if } i > \frac{h}{2}, \ \mbox{$l$ is odd} \\ [0.3cm]
2mk(i-1)(n-1)+2mk(j-1)+l+mk, & \mbox{ if } i \leq \frac{h}{2}, \ \mbox{$l$ is even}  \\ [0.3cm]
2mk(h-i)(n-1)+2mk(j-1)+l, & \mbox{ if } i > \frac{h}{2}, \ \mbox{$l$ is even}. 
\end{array}
\right.
\end{eqnarray*}

\textsf{Case B.}~$km$ is even. Let $w_{t} := w_{i,j}^{l}$, where 
\begin{eqnarray*}
t = &\left\{
\begin{array}{ll}
mk(i-1)(n-1)+mk(j-1)+l, & \mbox{ if } l \mbox{ is even}, \\ [0.3cm]
mk(h-i)(n-1)+mk(j-1)+l, & \mbox{ if } l \mbox{ is odd}.
\end{array}
\right.
\end{eqnarray*}

Next define an ordering $\vec{u} =  (u_0,u_1,\ldots,u_{p-1})$ as follows. 
If $|W(G)| = 1$, let $u_{p-1} = w^1$ and $u_i := w_{i+1}$ for $0 \leq i \leq p-2$. If $|W(G)| = m \geq 2$, let $u_0 = w^m$ and $u_i:=w_i$ for $1 \leq i \leq p-m$. For $p-m+1 \leq i \leq p-1$, let $u_i := w^{i-p+m}$.

It is easy to see that $\vec{u}$ defined above satisfies \cref{thm:lb} (a) and (b). 
To complete the proof it suffices to show $\vec{u}$ satisfies \eqref{eq:dij}  in \cref{thm:main}. If $d(G)$ = 2 or 3 then it is easy to verify \eqref{eq:dij}. Hence, assume $d(G) \geq 4$. Let $u_i$ and $u_j$ be two vertices with $0 \leq i < j \leq p-1$. Note that by our  definition of $\vec{u}$, for any $0 \leq i \leq p-2$, it holds that  $L(u_i) + L(u_{i+1}) \leq (d(G)+3+\ve(G))/2$,   and  the equality holds for at most one pair in every $mk$  consecutive terms. Denote the right-hand-side of \eqref{eq:dij} by $S_{i,j}$, then for $j \geq i+3$ we have: 
$$
\begin{array}{lll} 
S_{i,j} &:=& \sum\limits_{t=i}^{j-1}\left[L(u_t)+L(u_{t+1})\right]-(j-i)(d(G)+\ve(G))+1+d(G) \\
&\leq& (j-i-2)\(\frac{d(G)+3+\ve(G)}{2}\)+2\(\frac{d(G)+1+\ve(G)}{2}\)-2(d(G)+\ve(G))+1-\ve(G) \\
&=& (j-i-2)\left[3-d(G)-\ve(G)\right]/2+2-\ve(G) \\
&< & 2-\ve(G) \leq d(u_i,u_j). \end{array}
$$
If $j=i+2$, then $S_{i,j} \leq 3 -  \ve(G) \leq d(u_i,u_j)$, as $u_i$ and $u_j$ are in opposite or different branches. Thus \eqref{eq:dij} is satisfied. If $j=i+1$ it is straightforward to verify that \eqref{eq:dij} is satisfied. Therefore, by Theorem \ref{thm:main}, $G$ is a lower bound block graph. 
\end{proof}

For better understanding of the above results the readers are referred to Figure \ref{Fig:Starblock}.

%
\section{Radio numbers of the line graphs of trees}
%

Let $G$ be a graph. The {\it line graph} of $G$, denoted by $\mathcal{L}(G)$, is the graph whose vertices are the edges of $G$, where $ef \in E(\mathcal{L}(G))$ when $e \cap f \neq \emptyset$. Observe that every tree and its line graph are both block graphs. It is known and easy to see that a graph $G$ is the line graph of a tree if and only if $G$ is a connected block graph in which each cut-vertex is on exactly two blocks (cf.  \cite{Harary}). 

It is known that a connected graph $G$ is isomorphic to its line graph if and only if $G$ is a cycle. Hence there is a one-one correspondence between trees and their line graphs. Note that the line graph of an $n$-vertex path $P_n$ is $P_{n-1}$. It is known \cite{Daphne1} that $P_{m}$ is a lower bound tree if and only if $m$ is even.  Hence,  there exist lower bound trees whose line graphs are not lower bound block graphs, and  there exist non-lower bound trees whose line graphs are lower bound block graphs. 

In this section we investigate the relations between the radio numbers of a tree and its line graph. In particular, we establish conditions for lower bound trees so that their line graphs are lower bound block graphs. Moreover, we set up conditions so that the trees are lower bound trees given their line graphs are lower bound block graphs.  

We use the following presentation for the line graph of a tree. Let $T$ be a tree rooted at a weight center $w^*$. If there exists an edge $e$ = $uv$ with $d(w^*,v)$ = $d(w^*,u)+1$, then $u$ is called the \emph{edge ancestor} of $v$, and $v$ is an \emph{edge descendent} of $u$. We denote a vertex in $\mathcal{L}(T)$ by the edge descendent of the corresponding edge in $T$.  
That is, if $uv \in E(T)$ and $v$ is an edge descendent vertex of $u$  then we denote the edge $uv$ by the vertex $v$ in $\mathcal{L}(T)$. 

The following result is useful in our proofs. 
\begin{Theorem}
{\rm \cite{Daphne1}}
\label{weight center}
Every tree $T$ has one or two weight centers. If $w$ is a weight center of $T$, then each component of $T - w$ contains at most $|V(T)|/2$ vertices. Moreover, $T$ has two weight centers, $W(T)=\{w, w'\}$, if and only if $ww'$ is an edge of $T$ and $T-ww'$ consists of two equal-sized components. 
\end{Theorem}

Combining \cref{weight center} and definition of line graphs, we obtain 
\begin{Corollary} 
\label{two weight center} 
For a tree $T$ with $|W(T)|=2$, we have  $|W(\mt{L}(T))| = 1$. 
\end{Corollary}

Consider the case that $|W(T)|=|W(\mt{L}(T))| = 1$. Denote  $W(T)=\{w^*\}$ and $W(\mt{L}(T)) = \{w\}$. In $T$ (rooted at $w^*$), $w$ is an 
edge descendant of $w^*$. Define $B(T)$ the subgraph of $\L(T)$ induced by the following vertex set:
$$
V(B(T))=\{w\} \cup \{v \in V(T): v \ \mbox{is a descendant of} \ w \ \mbox{in} \ T\}.
$$

\begin{Lemma}\label{lem:et} Let $T$ be a tree with $|W(T)| = |W(\L(T))| = 1$. Denote $W(\L(T))=\{w\}$. Then 
\begin{enumerate}[\rm (i)]
    \item $\L(T)$ consists of two branches, $B(T) -  w$ and $\L(T) - B(T)$. 
    \item $|B(T)| \leq \lfloor \frac{|V(T)|-1}{2} \rfloor$. 
\end{enumerate}
\end{Lemma}

\begin{proof} Denote $W(T) = \{w^*\}$. Then $w$ is an edge descendant of $w^*$ in $T$ and there are two  
branches in $\L(T)$, namely, $B(T)-w$ and $\L(T) - B(T)$. Hence (i) is true. Because $|W(T)|=1$, by \cref{weight center}, each branch in  $\L(T)$ contains at most $  \frac{|V(T)|-3}{2}$ vertices. Thus (ii) follows.
\end{proof}

\begin{Observation}\label{line:obs} Let $T$ be a tree of order $p\geq 2$ and let  $\L(T)$ be its line graph. The following can be obtained without difficulty: 
\begin{enumerate}[\rm (i)]
\item $|V(\mt{L}(T))| = |V(T)|-1$, \   $\diam(\mathcal{L}(T)) = \diam(T)-1$. 
\item $d_{\L(T)}(u,v) = \left\{
\begin{array}{ll}
d_{T}(u,v)-1, & \mbox{none of $u$ and $v$ is a descendent of the other in $T$}, \\
d_{T}(u,v), & \mbox{otherwise.}
\end{array}
\right.$
\item $L_{\mathcal{L}(T)}(v) = \left\{
\begin{array}{ll}
L_{T}(v)-1, & |W(T)| = |W(\L(T))|=1 \mbox{ and } v \in V(B(T)), \mbox{  or } |W(\L(T))| \geq 2, \\
L_{T}(v), & \mbox{ otherwise. } \\
\end{array}
\right.$
\item $L(\mathcal{L}(T)) = \left\{
\begin{array}{ll}
L(T)-|B(T)|, & |W(T)| = |W(\L(T))| = 1, \\
L(T)-p+1, & |W(T)| = 1 \mbox{ and }|W(\L(T))| > 1,  \\
L(T), & |W(T)| = 2.
\end{array}
\right.$
\end{enumerate}
\end{Observation}

The next two results deal with the situation that $T$ is a lower bound tree.  Precisely,  we show necessary conditions that its block graph is also a lower bound block graph. 
\begin{Theorem}\label{thm:nolb1} Let $T$ be a tree with diameter $d(T) \geq 3$ and $|W(T)| = |W(\L(T))| = 1$. If $|V(T)|$ is even, or 
    $|V(T)|$ is odd with $|B(T)| < (|V(T)|-1)/2$, then $\L(T)$ is not a lower bound block graph.
\end{Theorem}
\begin{proof} Denote $W(\L(T))=\{w\}$. By the assumption and  \cref{lem:et}, we have $|B(T)| \leq \lfloor |V(T)|/2 \rfloor  - 1$, and $\L(T)$ consists of two branches, $B(T)-w$ and $\L(T) - B(T)$. Since $|V(B(T)) \setminus \{w\}| \leq \lfloor |V(T)|/2 \rfloor-2$, it must be $|V(\L(T)) \setminus V(B(T))| \geq \lfloor |V(T)|/2 \rfloor$. Therefore it is impossible to find an ordering of $V(\L(T))$ such that \cref{thm:lb} (b) holds. Hence, $\L(T)$ is not a lower bound block graph.   
\end{proof}

\begin{Theorem}\label{thm:line4} Let $T$ be a tree of order $p$ and  diameter $d(T) \geq 2$. Let $\L(T)$ be its line graph. 
Assume $T$ is a lower bound tree with an optimal ordering  $\vec{u} = (u_0,u_1,\ldots,u_{p-1})$ of $V(T)$ from  \cref{thm:lb} such that $L(u_{p-2})=1$, and one of the following holds: 
\begin{enumerate}[\rm (i)]
    \item $|W(T)|= |W(\L(T))| = 1$, $p$ is odd, and $|B(T)| = (p-1)/2$; \item $|W(T)| =1$ and $|W(\L(T))| > 1$;
    \item $|W(T)|=2$.
\end{enumerate}
Then $\L(T)$ is a lower bound block graph and 
$
rn(\L(T))=LB(\L(T))=rn(T)-d(T)+1 - \ve (T). 
$
\end{Theorem}
\begin{proof}
Assume $T$ is a lower bound tree and there exists an optimal  ordering  $\vec{u}=(u_0,u_1,\ldots,u_{p-1})$ of $V(T)$ satisfying \cref{thm:lb} such that $L(u_{p-2})=1$.  Without loss of generality, assume $u_0$ is a weight center of $T$. 

(i) $|W(T)| = |W(\L(T))| = 1$, $p$ is odd, and $|B(T)| = (p-1)/2$. Denote $W(\L(T)) = \{w\}$. Define an ordering $\vec{u'}=(u'_0, u'_1, \ldots, u'_{p-2})$ of $V(\L(T))$ as follows: $u'_0 = w$ and  $u'_i=u_i, 1 \leq i \leq p-2$. Then $\vec{u'}$ satisfies (a) and (b) in \cref{thm:lb} for $\L(T)$, and is alternating between the two branches, $B(T)-w$ and $\L(T) - B(T)$. It is enough to  show that $\vec{u'}$ satisfies \eqref{eq:dij} in \cref{thm:main}.  Let $u'_i$ and $u'_j$ be two arbitrary vertices, $0 \leq i < j \leq p-2$. Denote the right-hand side of \eqref{eq:dij} by $S_{i,j}$. Because $\vec{u}$ satisfies (4) for $T$, by  \cref{line:obs} (i) (ii) (iii) we obtain 
$$
\begin{array}{lll}
S_{i,j} &=& \sum\limits_{t=i}^{j-1}\left[ L_{\L(T)}(u'_t)+L_{\L(T)}(u'_{t+1})-(d(\L(T))+1) \right] +d(\L(T))+1 \\
&=& \sum\limits_{t=i}^{j-1}  \left[L_{T}(u'_t)-1+L_{T}(u'_{t+1})-d(T)\right]+d(T) \\ 
&=& \sum\limits_{t=i}^{j-1}\left[L_{T}(u'_t)+L_{T}(u'_{t+1})-(d(T)+1)\right]+d(T)+1-1 \\
&\leq& d_{T}(u'_i,u'_j)-1 \leq d_{\L(T)}(u'_i,u'_j).
\end{array}
$$

Hence $\L(T)$ is a lower bound block graph.  By Theorem \ref{thm:main} and \cref{line:obs},  
$$
\begin{array}{ccl}
    rn(\L(T)) & = &(p-2)(d(\L(T))+1)-2L(\L(T))+1  \\
     & = & (p-2)(d(T))-2(L(T)-|B(T)|)+1 \\
     & = & rn(T)-d(T) - p + 1 + 2|B(T)| 
      =  rn(T)-d(T).
 \end{array}
$$

(ii) $|W(T)|=1$ and $|W(\mt{L}(T))| > 1$. 
Define an ordering for $V(\L(T))=\vec{u'} =  (u'_0,u'_1,\ldots,u'_{p-2})$ by  $u'_i=u_i$, $0 \leq i \leq p-2$. Then $\vec{u'}$ satisfies (a) and (b) in \cref{thm:main} for $\L(T)$. 
Similar to (i), one can show that $\vec{u'}$ satisfies \eqref{eq:dij} in \cref{thm:main}. So $\L(T)$ is a lower bound block graph. 

Again, by Theorem \ref{thm:main} and \cref{line:obs} we get 
$$
\begin{array}{ccl}
    rn(\L(T)) & = &(p-2)(d(\L(T)))-2L(\L(T))  \\
     & = & (p-2)(d(T)-1)-2(L(T)-p+1) 
      =  rn(T)-d(T).
 \end{array}
$$

(iii) $|W(T)|=2$. 
In this case, $\L(T)$ has only one central vertex, say $w$. Define an ordering of $V(\L(T))=\vec{u'}=(u'_0,u'_1,\ldots,u'_{p-2})$ by $u'_0=w$ and $u'_i=u_i, 1 \leq i \leq p-2$. Then  $\vec{u'}$ satisfies \cref{thm:lb} (a) (b), and \eqref{eq:dij} in \cref{thm:main} for $\L(T)$. Thus, $\L(T)$ is a lower bound block graph.    
By Theorem \ref{thm:main} and \cref{line:obs}, 
$$ 
\begin{array}{ccl}
    rn(\L(T)) & = &(p-2)(d(\L(T))+1)-2L(\L(T))+1  \\
     & = & (p-2)d(T)-2L(T)+1 
      =  rn(T)-d(T)+1. 
\end{array} 
$$
The proof is complete. 
\end{proof}
The above results can be applied to the line graphs of many known  families of lower bound trees. 
A {\it level-wise regular tree} is a tree rooted at one vertex $w$ or two (adjacent) vertices $w$ and $w'$, in which all vertices with the minimum  distance $i$ from $w$ or $w'$  have the same degree $m_i$, for $0 \leq i < h$, where $h$ is the height of $T$. Denote these trees by $T^1 = T^1_{m_0,m_1,\ldots,m_{h-1}}$ (with one root) and $T^2 = T^2_{m_0,m_1,\ldots,m_{h-1}}$ (with two roots),  respectively. It was proved by Hal\'asz and Tuza \cite[Theorem 4-5]{Tuza}
that $T^1$ and $T^2$ are both lower bound trees when $m_i \geq 3$ for all $i$. 

A {\it complete $m$-ary tree} of height $h$ denoted by $T_{h,m}$ is a single-rooted complete level-wise regular tree with height $h$ where $m_0=m$ and $m_i=m+1$ for all $1 \leq i \leq h$. For instance, a complete binary tree is when $m=2$. In \cite[Theorem 13]{Li}, Li et al.  determined the radio number of $T_{h,m}$. In particular, the authors proved that $T_{h,m}$ are lower bound trees if and only if $m \geq 3$. Note that a complete $m$-ary tree with $m\geq 3$ is a special case for $T^1$ with $m_0 \geq 3$ and $m_i \geq 4$ for $0 \leq i \leq h$. 

The $(n,k)$-\emph{banana tree}, denoted by $B(n,k)$, is a single-rooted level-wise regular tree of height 3. Precisely: 
$$
B(n,k) = T^1_{n, 2, k-1}. 
$$
The $(n,k)$-\emph{firecrackers tree}, denoted by $F(n,k)$, is the tree obtained by taking $n$ copies of a $(k-1)$-star and identifying a leaf of each of them to a vertex of $P_n$ ($n$-vertex of path).  
In \cite[Theorem 4.1-4.2]{Bantva2}, Bantva et al. proved that banana trees $B(n, k)$ with $n \geq 5$ and $k \geq 4$ and firecracker trees $F(n,k)$ with $n, k \geq 3$ are all lower bound trees. Chavez et al. \cite{CLS} proved a more general result and showed that $B(n,k)$ is a lower bound tree for $n \geq 5$ and $k \geq 3$. 

See Figure 4 for examples of the above families of lower bound trees. Note that for every banana tree and complete level-wise regular tree  with $m_i \geq 3, 0 \leq i < h$, there exists an optimal radio labeling $f$ whose induced ordering  $\vec{u}=(u_0,u_1,\ldots,u_{p-1})$ satisfies $L(u_{p-2}) = 1$ where $p = |V(T)|$ (given in  \cite{Bantva2,Li,Tuza}). For the firecracker trees, it is easy to obtain an optimal radio labeling $f$ whose induced ordering $\vec{u}=(u_0,u_1,\ldots,u_{p-1})$ satisfies $L(u_{|V(G)|-2}) = 1$ from the optimal radio labeling given in \cite{Bantva2}. By Theorem \ref{thm:line4}, the line graphs of all these lower bound trees are lower bound block graphs. Hence, one can determine their radio numbers  by direct calculation.

\begin{figure}[ht]
\begin{center}
  \includegraphics[width=5.5in]{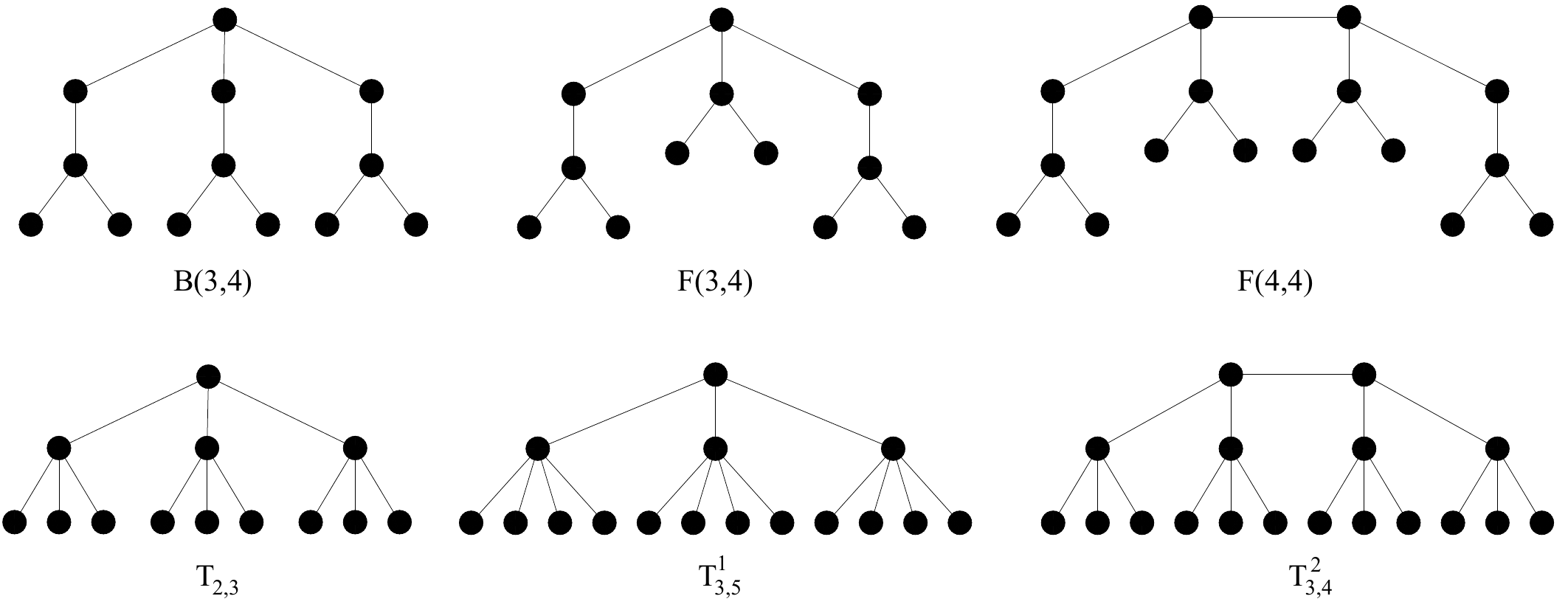}\\
\caption{Banana tree, firecrackers tree, complete $3$-ary tree and level-wise regular tree.}
\label{Fig:Trees}
\end{center}
\end{figure}

\begin{Corollary}
\label{banana:thm} Let $n \geq 5$ and $k \geq 4$ be integers. Then $\L(B(n,k))$ is a lower bound block graph and 
$
rn(\L(B(n,k))) = n(k+6)-5.
$
\end{Corollary}

\begin{Corollary}
\label{fire:thm} Let $n,k \geq 3$ be integers. Then $\L(F(n,k))$ is a lower bound block graph and 
$$
rn(\L(F(n,k))) = \left\{
\begin{array}{ll}
\frac{(n^{2}+1)k}{2}+4n-6, & \mbox{ if } n \mbox{ is odd}, \\ 
\frac{n^{2}k}{2}+4n-5, & \mbox{ if } n \mbox{ is even}.
\end{array}
\right.
$$
\end{Corollary}

\begin{Corollary} \label{level:thm} If $m_i \geq 3$ for all $0 \leq i < h$, then $\L(T^1)$ and $\L(T^2)$  are lower bound block graphs and 
$$ 
rn(\L(T^1)) = (d+1)(n-1)+1-2\displaystyle\sum_{i=1}^{h}\(m_0 \cdot i \cdot \displaystyle\prod_{0<j<i}(m_j-1)\)-2h, 
$$
$$
rn(\L(T^2)) = d(n-1)-4\displaystyle\sum_{i=1}^{h}\(i \cdot \displaystyle\prod_{j=0}^{i-1}(m_j-1)\)-2h, 
$$
where $d=2h$ (for $T^1$) and $d=2h+1$ (for $T^2$).
\end{Corollary}

In the following three results, we prove conditions such that $T$ is a lower bound tree, given its line graph is a lower bound block graph. 

\begin{Theorem}
Let $T$ be a tree with diameter $d(T)$ and $W(T)=\{w,w'\}$. If its line graph $\L(T)$ is a lower bound block graph, then $T$ is a lower bound tree and  $rn (T) = rn(\L(T)) + d(T) -1$. 
\end{Theorem}

\begin{proof} By \cref{weight center}, $T - ww'$ has two equal-sized components and $|V(T)|=p$ is even. By \cref{two weight center}, $|W(\L(T)|=1$.  Denote $W(\L(T))=\{w^*\}$ (which is either $w$ or $w'$). Then $\L(T)-w^*$ has exactly two branches, denoted by $B_{x}$ and $B_{y}$, each has $(p-2)/2$ vertices.  

Because $\L(T)$ is a lower bound block graph, by \cref{thm:lb}, there exists  an ordering of $V(\L(T)) = \vec{u} = (u_0,u_1, \ldots,u_{p-2})$ such that \cref{thm:lb} holds where $f$ is the labeling defined in \cref{thm:lb} (c). By symmetry, assume $u_0=w^*$ and $u_{p-2}$ is a vertex of level 1 in $\L(T)$. Without loss of generality, assume $u_{p-2} \in B_{x}$.  Then $u_{1} \in B_{y}$.   

Define an ordering $\vec{u'}$ of $V(T)$ by: $u'_0 = w, u_{p-1}'=w'$, and $u'_i = u_i$ for $1 \leq i \leq p-2$.  Then (a) and (b) in \cref{thm:lb} hold for $T$. Define a labeling $f'$ on $V(T)$ by $f'(u'_i)=f(u_i)$ for $0 \leq i \leq p-2$, and $f'(u'_{p-1}) = f(u_{p-2}) + d(T)-1$. By the facts that $d(\L(T))=d(T)-1$, $L_{\L(T)}(u_{p-2}) = 1$, and for all $0 \leq i \leq p-1$, it holds that $d_T(u_i, u_{i+1}) = L_{\L(T)} (u_i) + L_{\L(T)} (u_{i+1}) + \ve(\L(T)) = L_T(u_i)+ L_T(u_{i+1}) + 1$, we conclude that $f'$ is a radio labeling for $T$ that satisfies  \cref{thm:lb} (c). Therefore the proof is complete.  
\end{proof}

\begin{Theorem} Let $T$ be a tree with $|V(T)|=p$ for some odd $p$ and diameter $d(T)$. If $|W(T)| = |W(\L(T))| = 1$ and $\L(T)$ is a lower bound block graph, then $T$ is a lower bound tree and $rn(T) = rn(\L(T))+d(\L(T))+1.$ 
\end{Theorem}

\begin{proof} Denote $W(T)=\{w^*\}$ and $W(\L(T))=\{w\}$. Then $w$ is an edge descendant of $w^*$ in $T$, and $w$ is a cut-vertex of exactly two  
branches in $\L(T)$, namely, $B(T)-w$ and $\L(T) - B(T)$. Denote these two branches by $B$ and $B'$, respectively. 

As $\L(T)$ is a lower bound block graph, by  
\cref{thm:nolb1}, $|B(T)| = (p-1)/2$. By \cref{thm:lb}, there exists an ordering $\vec{u} = (u_0,u_1,\ldots,u_{p-2})$ of $V(\L(T))$ such that \cref{thm:lb} holds, where $f$ is the labeling in \cref{thm:lb} (c). By symmetry, assume $u_0 = w$. Note that in $\vec{u}$, the vertices with odd (or even,  respectively) subscripts are in $B'$ ($B$, respectively).

Define an ordering $\vec{u'} =  (u'_0,u'_1,\ldots,u'_{p-2})$ of $V(T)$ by: $u'_0 = w^*$ and $u'_i = u_{i-1}$ for $1 \leq i \leq p-1$. Then (a) and (b) in \cref{thm:lb} hold for $T$. Define a labeling $f'$ on $V(T)$ by 
$f'(u'_0)=0$ and $f'(u'_{i+1}) = f(u_i)+d(\L(T))+1$  
for all $0 \leq i \leq p-2$. 
One can easily verify that $f'$ is a radio labeling of $T$  satisfying (c) in \cref{thm:lb}  with the desired span. The proof is complete. 
\end{proof} 

\begin{Theorem} Let $T$ be a tree of order $p$ and diameter $d(T)$. Assume $|W(\L(T))| \geq 2$, $\L(T)$ is a lower bound block graph, and there exists an ordering $\vec{u} = (u_0,u_1,\ldots,u_{p-2})$ of $V(\L(T))$ satisfying \cref{thm:lb} such that the following holds for all $0 \leq  i < j \leq p-2$: If $u_i$ or $u_j$ is a descendent of the other in $\L(T)$ then $|j-i| \geq d(\L(T))+1$.  
Then $T$ is a lower bound tree, and $rn(T) = rn(\L(T))+d(\L(T))+1.$ 
\end{Theorem}
\begin{proof} Since $|W(\L(T))| \geq 2$, so $|W(T)|=1$. Denote $W(T)=\{w^*\}$ and $W(\L(T))$ = \{$w^1$, $w^2$, $\ldots$, $w^t$\} for some $t \geq 2$. 

Define an ordering $\vec{u'}$ of $V(T)$ by $u'_0 = w$ and $u'_i = u_{i-1}$ for $1 \leq i \leq p-1$. Then (a) and (b) in \cref{thm:lb} hold for $T$. 

Define a labeling $f'$ on $V(T)$ by $f'(u'_{0}) = 0$ and $f'(u'_{i+1}) = f(u_i)+d(\L(T))+1$ for all $0 \leq i \leq p-2$. It is enough to prove that $f'$ is a radio labeling for $T$. 
Let $u'_i$ and $u'_j$ be two arbitrary vertices, $0 \leq i < j \leq p-1$. If $i=0$ then it is easy to see that $f'(u'_j)-f'(u'_0) \geq d(T)+1-d_{T}(u'_0,u'_j)$.
Assume that $i \neq 0$. If $j-i \geq d(\L(T))+1$ then $f'(u'_j)-f'(u'_i) = f(u_j) - f(u_i) \geq j-i \geq d(\L(T))+1 \geq d(T) + 1 - d_{T}(u'_i,u'_j)$ as  $d_{T}(u'_i,u'_j) \geq 1$. If $j-i < d(\L(T))+1$ then by the hypothesis, none of $u'_i,u'_j$ is a descendent of the other. By \cref{line:obs} (ii), $d_{\L(T)}(u'_i,u'_j) = d_{T}(u'_i,u'_j)-1$. Therefore, $f'(u'_j)-f'(u'_i) = f(u_j) - f(u_i) \geq d(\L(T))+1 - d_{\L(T)} (u'_i, u'_j) \geq d(T) + 1 - d_{T}(u'_i,u'_j)$. Hence $f'$ is a radio labeling of $T$ with the desired span. 
\end{proof}

While we have shown several families of trees such that both the trees and their corresponding line graphs are lower bound block graphs, there  are trees for which $T$ is not a lower bound block graph but its line graph is a lower bound block graph. That is, $rn(\L(T)) = LB(\L(T))$ but $rn(T) \neq LB(T)$. For instance, $\L(P_{2k+1}) = P_{2k}$ and we know $P_{2k+1}$ is not a lower bound tree while  $P_{2k}$ is. In addition, we give another example as follows.  A tree is called a \emph{caterpillar} if the removal of all its degree-one vertices results in a path, called the spine. Denote by $C(n,k)$ the caterpillar in which the spine has length $n-3$ and all vertices on the spine have degree $k$, where $n \geq 3$ and $k \geq 3$. It was shown by Vaidya and Bantva \cite{Vaidya3} that $\L(C(n,k))$ is a lower bound block graph. On the other hand, it was proved by Bantva et al. \cite{Bantva2} that $rn(C(n,k)) = LB(C(n,k))$ when $n=2m$, but $rn(C(n,k))=LB(C(n,k))+1$ when $n=2m+1$.  

On the other hand, so far, all the known lower bound trees excepts paths have their line graphs being lower bound block graphs.  The following question is raised:

\begin{Question}
Are there infinitely many lower bound trees $T$ other than paths for which $\L(T)$ is not a lower bound block graph? 
\end{Question}

\bigskip
\noindent\textbf{Acknowledgement.}  The authors are grateful to the three  anonymous referees for their careful reading of this article and for their insightful and helpful comments.

\end{document}